\documentclass[11pt]{amsart}
\usepackage[utf8]{inputenc}
\usepackage[margin=3cm]{geometry}
\RequirePackage[colorlinks,citecolor=blue,urlcolor=blue]{hyperref}
\usepackage{comment,graphicx}
\usepackage{float}

\newtheorem{thm}{Theorem}
\newtheorem{lemma}[thm]{Lemma}
\newtheorem{prop}[thm]{Proposition}
\newtheorem{cor}[thm]{Corollary}
\newtheorem{rem}[thm]{Remark}
\newtheorem{example}{Example}[section]

\newcommand{\p}{{\mathbb P}}
\newcommand{\e}{{\mathbb E}}
\newcommand{\D}{{\mathrm d}}
\newcommand{\bs}{\boldsymbol}
\renewcommand{\a}{\alpha}
\newcommand{\ba}{{\bs \alpha}}

\newcommand{\R}{{\mathbb R}}
\newcommand{\C}{{\mathbb C}}

\newcommand{\diag}{{\rm diag}}
\renewcommand{\i}{{\rm i}}

\newcommand{\1}[1]{\mbox{\rm\large  1}_{\{#1\}}}

\newcommand{\mbcomment}[1]{\begin{list}{}{\leftmargin\parindent\rightmargin0pt\listparindent0pt\parsep1ex\labelwidth0pt\labelsep0pt}\item[]\small\sf${\color{red}\clubsuit\clubsuit\clubsuit}$~#1 \end{list}}

\newcommand{\mat}[1]{\boldsymbol{\bs #1}}
\newcommand{\vect}[1]{\boldsymbol{\bs #1}}

\begin{document}
\title[L\'evy process on a matrix-exponential time interval]{Fluctuation theory for one-sided L\'evy processes with a matrix-exponential time horizon}
\author[M.\ Bladt and J.\ Ivanovs]{Mogens Bladt and Jevgenijs Ivanovs}
\address{University of Copenhagen and Aarhus University}
\begin{abstract}
There is an abundance of useful fluctuation identities for one-sided L\'evy processes observed up to an independent exponentially distributed time horizon.
We show that all the fundamental formulas generalize to time horizons having matrix exponential distributions, and the structure is preserved.
Essentially, the positive killing rate is replaced by a matrix with eigenvalues in the right half of the complex plane which, in particular, applies to the positive root of the Laplace exponent and the scale function.
Various fundamental properties of thus obtained matrices and functions are established, resulting in an easy to use toolkit.
An important application concerns deterministic time horizons which can be well approximated by concentrated matrix exponential distributions.
Numerical illustrations are also provided.
\end{abstract}

\maketitle

\section{Introduction}
A spectrally-negative L\'evy process is a natural generalization of the classical Cram\'er-Lundberg risk process. 
This class of processes allows for a rich fluctuation theory, which has been extensively applied to various insurance risk and financial models. 
This theory, see~\cite{avram_review} for a recent review, is crucial in the analysis of risk processes in a wide range of settings including alternative ruin concepts, dividends, capital injections, loss-carry-forward taxation, and  optimal control. An important assumption here is that the time horizon $T$ is either infinite or is given by an independent exponential random variable.
Its rate $q\geq 0$ is often referred to as the killing rate of the L\'evy process.

In applications a deterministic time horizon is often needed. In this case Erlangization~\cite{usabel} is a popular approach, where the deterministic time is approximated using the Erlang distribution.
In fact, more general phase-type (PH) distributions can be used when the interest is in a random time horizon,
and the analysis often goes via fluid flow models and Markov additive processes~\cite{badescu, WH_PH}.
Matrix exponential (ME) distributions, equivalently distributions with rational transforms, form yet a broader class which has been successfully employed in modeling claim sizes~\cite{mordecki}.
Some preliminary investigations concerning ME time horizons have been carried out in~\cite[Ch.\ 6]{oscar}.
Importantly, there exist much more concentrated ME distributions as compared to PH (read Erlang) distributions for the same order~\cite{miklos}. 
Such concentrated ME distributions have been recently suggested by~\cite{horvath2020numerical} for a general  purpose numerical transform inversion method, and in various scenarios it has outperformed some classical methods. 

In this paper we provide a comprehensive fluctuation theory for spectrally-negative L\'evy processes with an independent ME time horizon~$T$.
Importantly, the classical formulas preserve their structure in this generalized context.
The basic idea is to analytically continue a given formula in the rate $q$ to the right half of the complex plane and then apply it to $-\mat T$, the negative of the ME generator matrix, using functional matrix calculus.
In particular, we extend the positive root $\Phi(q)$ of the Laplace exponent and the scale matrix $W_q(x)$ which play a fundamental role in fluctuation identities, and establish their essential properties.  
We also consider the Wiener-Hopf factorization and generalize various known formulas for the supremum/infimum and the terminal value.
Further identities can be extended in a rather straightforward way using the results and properties of this paper.
In the common case when the ME generator is diagonalizable, an ME formula can be seen as a linear combination of the respective classic formulas but for possibly complex killing rates.
This links to the inversion method of~\cite{horvath2020numerical} when their concentrated ME distributions are used.

Our theory brings new insights also for PH time horizons. In particular, we show that $-\Phi(-\mat T)$ and $W_{-\mat T}(x)$ correspond to some fundamental objects in the theory of Markov additive processes, where the same level process  is used in each phase. These convenient representations seem to be new. Furthermore, our analytic method results in neat formulas avoiding somewhat cumbersome time-reversed quantities as in~\cite{WH_PH}.

The remainder of this paper is organized as follows. This introductory section is concluded by providing some basic notation for ME distributions and L\'evy processes. In Section \ref{sec:toolkit} we present the general idea and the essential tools from the theory of functional matrix calculus. 
The next three sections deal with the natural extension of the fundamental fluctuation theory including scale functions and a variety of exit problems.
Section~\ref{sec:WH} presents the analogue of the classical Wiener--Hopf factorization and a suite of related identities.
The next two sections contain examples and numerics, respectively. We conclude by discussing an open problem concerning L\'evy processes observed at the epochs of an independent renewal process with ME inter-arrival times.


\subsection{Matrix-exponential distributions}
A positive random variable $T$ is said to have a matrix--exponential (ME) distribution if it has a density on the form
\[   f_T (x) = \mat{\alpha}e^{\mat{T}x}\mat{t},\qquad x\geq 0, \]
where $\mat{\alpha}$ is a row vector (bold Greek) with $p$ elements, $\mat{T}$ is a $p\times p$ matrix called the ME generator, and $\mat{t}$ is a column vector (bold Roman); all being real valued. 
According to~\cite[Thm. 4.2.9]{bladt_nielsen_book} it is always possible to choose a (canonical) representation $(\ba,\mat T,\mat t)$ satisfying $\mat{t}=-\mat{T}\mat{1}$ and hence $\ba\mat 1=1$, where $\mat 1$ is a vector of ones.
Nevertheless, it may be convenient to allow for an arbitrary ME representation and so we write $T \sim \mbox{ME}(\mat{\alpha},\mat{T},\mat t)$. 
We also define a column vector $\mat l=(-\mat T)^{-1}\mat t$, which is $\mat 1$ in the case of a canonical representation (this vector often appears as the $\mat l$ast vector in the formulas below).

The eigenvalues of $\mat{T}$ must have strictly negative real parts,
and the Laplace transform of $T$ is a rational function:
\[  \e e^{-s T} = \mat{\alpha}(s\mat{I}-\mat{T})^{-1}\mat{t},\qquad s\geq 0, \]
where $\mat I$ is the identity matrix.
The opposite implication is also valid: any positive distribution with a rational Laplace transform is an ME distribution,~\cite[Thm. 4.1.17]{bladt_nielsen_book}. 

The class of ME distributions is strictly larger than the class of phase-type (PH) distributions, where $\ba$ and $\mat T$ can be seen as a probability vector and a sub-intensity matrix of a transient Markov chain~$(J_t)_{t\geq 0}$.
The density of the life time of this chain $J$ started according to $\ba$ coincides with~$f_T$. In this case we write $T\sim\mbox{PH}(\ba,\mat T)$.
As mentioned before, a much better approximation of a positive deterministic number can be achieved using ME distributions as compared to Erlang or any other PH distribution for a fixed number of dimensions~$p$, see~\cite{miklos}.

\subsection{L\'evy processes}
Consider a L\'evy process $(X_t)_{t\geq 0}$, that is, a process with stationary and independent increments.
We assume  that $X$ is a spectrally-negative L\'evy process, meaning that $X$ has no positive jumps and its paths are not non-increasing a.s.
Such a process is characterized by the L\'evy-Khintchine formula providing an expression for the Laplace exponent $\psi(\theta)$:
\[\e e^{\theta X_t}=e^{\psi(\theta)t}, \qquad \psi(\theta)=\frac{1}{2}\sigma^2\theta^2+\gamma\theta+\int_{-\infty}^0\left(e^{\theta x}-1-\theta x\1{|x|<1}\right)\nu(\D x)\]
with $\Re(\theta)\geq 0$ and parameters $\gamma\in\R,\sigma\geq 0,\nu(\D x)$, where the latter is a measure satisfying $\int_{-\infty}^0 (x^2\wedge 1)\nu(\D x)<\infty$. 

Define the running supremum $\overline X_t=\sup_{s\leq t} X_s$ and analogously the running infimum $\underline X_t=\inf_{s\leq t}X_s$ for $t\geq 0$. The first passage times are denoted by
\[\tau_x^\pm=\inf\{t\geq 0: \pm X_t>\pm x\},\qquad x\geq 0.\]
In some cases we start the L\'evy process at $x\neq 0$ and then we write $\p_x$ and $\e_x$ to signify the starting value.

\section{The toolkit}\label{sec:toolkit}
\subsection{The basic idea}
We start with a simple observation which is fundamental for this work.
Let $X_{[0,t)}=(X_s)_{s\in[0,t)}$ for $t> 0$ denote a restriction of the path of $X$ to the time interval $[0,t)$. One may think of sending the path at time $t$ to some isolated absorbing state. 
 For an independent exponential random variable $e_q$ of rate $q>0$ and an arbitrary (measurable) non-negative functional $\phi$ there is the identity:
\begin{equation}\label{eq:tr0} g(q):=\e\phi(X_{[0,e_q)})=q\int_0^\infty e^{-q s}\e\phi(X_{[0,s)})\D s,\end{equation}
which is essentially the Laplace transform of $\e\phi(X_{[0,t)})$ in time variable.
Importantly, for a one-sided L\'evy process $X$ and a variety of functionals~$\phi$, this expression is explicit  (often in terms of the scale function defined below), see~\cite{avram_review} and references therein for a long list of formulas. The main reason is that killing at $e_q$, by the memory-less property,  preserves stationarity and independence of increments. 

Letting $T\sim\mbox{ME}(\ba,\mat T,\mat t)$ be independent of $X$ we get a similar matrix-form expression:
\[\e\phi(X_{[0,T)})=\ba \Big(\int_0^\infty e^{\mat T s}\e\phi(X_{[0,s)})\D s\Big) \mat t.\]
Laplace transforms are analytic in their region of convergence, so the function in~\eqref{eq:tr0} is analytic for $q\in \C_+=\{z\in\C:\Re(z)>0\}$. Since
the eigenvalues of $-\mat T$ are all in $\C_+$, \cite{bladt_nielsen_book},  we may then apply $g$ to the matrix $-\mat T$, as we next explain in Section \ref{subsec:matrix-calc}, to obtain 
\begin{equation}\label{eq:tool}\e\phi(X_{[0,T)})=\ba g(-\mat T)(-\mat T)^{-1}\mat t=\ba g(-\mat T)\mat l,\end{equation}
where $\mat l=(-\mat T)^{-1}\mat t$ is $\mat 1$ for a canonical representation. 
The main difficulty is to analytically continue the formula corresponding to $g(q)$ to the half plane $\C_+$, and to establish some important properties of the involved components.

Furthermore, we allow $T$ to be a defective ME in the sense that the total mass $\p(T\in[0,\infty))=\ba\mat l\in (0,1]$ may be strictly less than~$1$. 
It is still true that the eigenvalues of $-\mat T$ are in $\C_+$. The basic example is $\mat T=\widehat {\mat T}-\delta\mat I$ for some $\delta>0$ and a non-defective $\widehat T\sim\mbox{ME}(\ba,\widehat{\mat T},\mat t)$, which allows to incorporate discounting into the formula:
\[\e\phi(X_{[0,T)})=\ba \Big(\int_0^\infty e^{-\delta s}e^{\widehat{\mat  T} s}\e\phi(X_{[0,s)})\D s\Big) \mat t=\e\big(e^{-\delta \widehat T}\phi(X_{[0,\widehat T)})\big).\]
Here we have independent exponential killing of rate $\delta>0$, whereas in phase-type setting it is natural to consider phase-dependent killing allowing to weigh time spent in each phase differently, see~\cite{ivanovs_killing} for further details and applications.
A somewhat related useful trick corresponds to considering $T=\widehat T\wedge e_\delta$ which has $\mbox{ME}(\ba,\widehat{\mat T}-\delta\mat I,\mat t+\delta\mat l)$ distribution, 
see Remark~\ref{rem:addTr}.

\subsection{Matrix calculus}\label{subsec:matrix-calc}
For a function $f:\C\mapsto \C$ analytic on some open connected domain $D\subset\C$ and a matrix $\mat{M}\in \C^{p\times p}$ {\it with eigenvalues in $D$} it is standard to define $f(\mat{M})$ via Cauchy's integral formula
\[f(\mat{M})= \frac{1}{2\pi\i}\oint_\gamma f(z)(z\mat{I}-\mat{M})^{-1} \D z,\]
where the closed simple curve $\gamma\in D$ contains the eigenvalues of $\mat{M}$ in its interior~\cite{matrixFunctions}. An equivalent definition of $f(\mat{M})$ proceeds via the Jordan canonical form of~$\mat{M}$.
In particular, in the diagonalizable case $\mat M=\mat P\diag_k(\lambda_k)\mat P^{-1}$ we have $f(\mat M)=\mat P\diag_k(f(\lambda_k))\mat P^{-1}$. 
In fact, both definitions can be given without requiring $D$ to be connected.
In the following we mostly work with the domain $\C_+$ and the matrix $-\mat T$, where $\mat T$ is an ME generator.
For convenience we state some basic calculus rules, see~\cite[Ch.\ 1]{matrixFunctions}.
\begin{lemma}\label{lem:calculus}
Let $\mat{M}\in\C^{p\times p}$ be a matrix and $f,g$ two analytic functions on their respective domains. Under the condition on eigenvalues belonging to the respective domain we have
\begin{itemize}
\item $f(\mat M)$ commutes with $g(\mat M)$,
\item the eigenvalues of $f(\mat M)$ are given by $f$ applied to the eigenvalues of $\mat M$,
\item $f(\mat{M})+g(\mat{M})=(f+g)(\mat{M})$ and $f(\mat{M})g(\mat{M})=(fg)(\mat{M})$,
\item $f(g(\mat{M})) = (f\circ g)(\mat{M})$,
\item for a power series $f(z)=\sum_{k=0}^\infty a_kz^k$ convergent on some domain it holds that \[f(\mat M)=\sum_{k=0}^\infty a_k\mat M^k,\]
\item for a $\sigma$-finite measure~$\mu$ and $f(z) = \int_0^\infty e^{z t}\mu (\D t)$ convergent on some domain it holds that \[f(\mat{M})=\int_0^\infty e^{\mat{M} t}\mu(\D t).\]
\end{itemize}  
\end{lemma}
The classic example of a matrix exponential is, of course, consistent with this theory.


\subsection{The diagonalizable case and relation to an inversion method}
Here we assume that $\mat T$ is diagonalizable: 
\begin{equation}\label{eq:diagonalizable}-\mat T=\mat P\diag_k\big(\lambda_k\big)\mat P^{-1},\end{equation} where $\lambda_k\in\C_+$ are the eigenvalues of~$-\mat T$. Hence the $\mbox{ME}(\ba,\mat T,\mat t)$ density $f(x)$ is a sum of exponential terms:
\begin{equation}\label{eq:fexp}f(x)=\ba \mat P\diag_k\big(e^{-\lambda_k x}\big)\mat P^{-1}\mat t=\sum_{k=1}^p c_k e^{-\lambda_k x}.\end{equation}
It is easy to see that every such density corresponds to some ME representation with a diagonalizable exponent matrix.
According to~\eqref{eq:tool} we then have
\[\e\phi(X_{[0,T)})=\ba \mat P\diag_k\big(g(\lambda_k)/\lambda_k\big)\mat P^{-1}\mat t=\sum_{k=1}^p c_k  \frac{g(\lambda_k)}{\lambda_k}.\]
Here again it is necessary to have a formula for $g$ anlytically continued to $\C_+$, and this requires proving some basic properties of the involved terms.

Taking an ME distribution concentrated around some deterministic value, say~1, we may obtain an approximation of $\e\phi(X_{[0,1)})$.
In~\cite{miklos} one particular family of such distributions was explored. 
They provided an efficient numerical procedure for determining the respective parameters $(c_k,-\lambda_k)$ in~\eqref{eq:fexp}, and published online a list of such parameters up to the order $p=1000$. 
As demonstrated  by~\cite{horvath2020numerical} this leads to a general procedure for numerical inversion of Laplace transforms.

The diagonalizable case may initially seem much simpler.
With the right perspective and tools, however, the general case is not more difficult as we demonstrate below.
Nevertheless the diagonalizable case is convenient for numerics and in some other special situations, see also Section~\ref{sec:MEjumps}.

\section{The supremum and first passage}
The first passage process is crucial to the analysis of fluctuations of~$X$.
By the strong Markov property of $X$ and absence of positive jumps we see that $(\tau_x^+)_{x\geq 0}$ is a (possibly killed) subordinator, that is a L\'evy process with non-decreasing paths a.s.
Here we also use the well known fact $\tau_0^+=0$ a.s.
Let $-\Phi(-q)$ be its Laplace exponent:
\begin{equation}\label{eq:taux}\p(\tau_x^+<e_q)=\e e^{-q \tau_x^+}=e^{-\Phi(q)x},\qquad x\geq 0,q> 0.\end{equation}  
It  has the following L\'evy-Khintchine representation:
\[\Phi(q)=\overline q+\overline d q-\int_0^\infty(e^{-qx}-1)\overline\nu(\D x),\qquad q\in \C_+,\]
where the parameters satisfy $\overline q,\overline d\geq 0, \int_0^\infty(1\wedge x)\overline \nu(\D x)<\infty$.
Note that the class of functions with such a representation coincides with the class of Bernstein functions~\cite[Thm.\ 3.2]{schilling},
 and it is closely related to completely monotone functions. The parameters can be identified from the basic fact that  
  $\Phi(q)$ is the unique positive solution of $\psi(\cdot)=q$ for $q>0$~\cite[Sec.\ VII]{bertoin_book}.
  
 The above L\'evy-Khintchine representation of $\Phi(q)$ is analytic on $\C_+$ and hence we may apply it to  a matrix $-\mat T$ with eigenvalues in $\C_+$ to get
\begin{equation}\label{eq:phi}
\Phi(-\mat T)=\overline q\mat I-\overline d \mat T-\int_0^\infty(e^{\mat Tx}-\mat I)\overline\nu(\D x).
\end{equation}
Here we use Lemma~\ref{lem:calculus} and its slightly generalized version in regard to the integral.
 Now we can easily apply the basic idea in~\eqref{eq:tool} to the functional $\phi(X_{[0,t)})=\1{\tau_x^+<t}$ to get an explicit formula for $\p(\tau_x^+<T)$ with an independent $T\sim\mbox{ME}(\ba,\mat T,\mat t)$.
In fact, quite a bit more can be said. In this regard we state the following Lemma which is proven in Appendix.
\begin{lemma}\label{lem:Phipsi}
For any $q\in \C_+$ it holds that $\Phi(q)\in \C_+$ and $\psi(\Phi(q))=q$. Moreover, for $\theta\in\C_+$ with $\psi(\theta)\in \C_+$ it must be that $\Phi(\psi(\theta))=\theta$. 
In particular, $\Phi(q)$ is the unique solution of $\psi(\cdot)=q$ in $\C_+$ for any $q\in \C_+$.
\end{lemma}
This result can also be useful when directly considering the diagonalizable case in~\eqref{eq:fexp}.
We are now ready to prove our main result concerning the supremum and first passage time for an ME time-horizon.
\begin{thm}\label{thm:Phi}
Let $T\sim\mbox{ME}(\mat{\alpha},\mat{T},\mat t)$ be independent of $X$.
Then 
\[\p(\overline X_T>x)=\p(\tau_x^+<T)=\ba e^{-\Phi(-\mat{T})x}\mat{l},\qquad x\geq 0,\]
which amounts to
\[ \overline X_T\sim \mbox{ME}(\ba,-\Phi(-\mat T),\Phi(-\mat T)\mat l). \]
Moreover, $\Phi(-\mat{T})\in \R^{p\times p}$ given in~\eqref{eq:phi} solves the equation $\psi(\cdot)=-\mat{T}$, 
and it is the unique solution among $\C^{p\times p}$ matrices with eigenvalues in~$\C_+$. 
\end{thm}
\begin{proof}
The equivalence between the probabilities is standard and follows form the fact that $X$ does not jump at~$T$ a.s.
The formula follows from~\eqref{eq:tool} and the representation $\p(\tau_x^+<e_q)$ of the transform, see~\eqref{eq:taux}.
Differentiate to obtain the \mbox{ME} density of $\overline X_T$ and thereby the stated ME representation.

From Lemma~\ref{lem:Phipsi} we see that the eigenvalues of $\Phi(-\mat{T})$ are all in $\C_+$ and according to Lemma~\ref{lem:calculus}
\[\psi\Big(\Phi(-\mat{T})\Big)=\psi\circ\Phi(-\mat{T})=-\mat{T}.\]
Finally, we prove uniqueness by considering a $p\times p$ matrix $\mat M$ with eigenvalues in $\C_+$ such that $\psi(\mat M)=-\mat T$.
Apply $\Phi$ to get $\Phi\circ\psi(\mat M)=\Phi(-\mat T)$, and so it is left to note that $\Phi(\psi(\theta))=\theta$ for some neighborhoods of the eigenvalues of $\mat M$, see also~\cite[Thm.\ 1.14]{matrixFunctions}.
\end{proof}


The case $T\sim \mbox{PH}(\ba,\mat T)$ is well understood, as it can be analyzed in the framework of a Markov additive processes.
Consider the Markov chain~$(J_{\tau_x^{+}})_{x\geq 0}$ formed by tracking the original phase $J$ at first passage times, and let $\mat G$ be its subintensity matrix.
It is now immediate that the supremum $\overline X_T$ is $PH(\ba,\mat G)$. The following representation of $\mat G$ implied by Theorem~\ref{thm:Phi} seems to be new.
\begin{cor}\label{cor:G}
For $T\sim \mbox{PH}(\ba,\mat T)$ the intensity matrix of the Markov chain $(J_{\tau_x^+}),x\geq 0$ is given by $\mat G=-\Phi(-\mat T)$.
\end{cor}
\begin{proof}
It is well known that $\mat G$ is the unique solution of $\psi(-\mat G)=-\mat T$ among $p\times p$ matrices with eigenvalues in $\C_+$, see e.g.~\cite[Thm.\ 2]{jordan} and also~\cite[Thm.\ 4.14]{thesis} allowing for phase-dependent killing rates and not imposing irreducibility requirements.
Now Theorem~\ref{thm:Phi} shows that $\mat G=-\Phi(-\mat T)$.
\end{proof}
The fact that $-\Phi(-\mat T)$ is a sub-intensity matrix of some Markov chain also follows from the general theory for Bernstein functions in \cite{schilling,berg_1993}.


\section{The scale function and two-sided exit}
The so-called scale function $W_q:[0,\infty)\mapsto [0,\infty)$ defined for all killing rates $q\geq 0$ plays a fundamental role in the fluctuation theory of one-sided L\'evy processes. 
It is the unique continuous, non-decreasing function identified by its transform
\begin{equation}\label{eq:transform}\int_0^\infty e^{-\theta x}W_q(x)\D x=1/(\psi(\theta)-q),\qquad \theta>\Phi(q),\end{equation}
see~\cite[Thm.\ VII.8]{bertoin_book}.
The scale function is strictly positive for $x>0$ and it solves the basic two-sided exit problem:
\begin{equation}\label{eq:exit}\e(e^{-q\tau_x^+};\tau_x^+<\tau_{-y}^-)=W_q(y)/W_q(x+y),\qquad x,y\geq 0,x+y>0.\end{equation}

It is well known that for any $x\geq 0$ the function $q\mapsto W_q(x)$ may be analytically continued to $q\in\mathbb C$ via the identity
\begin{equation}\label{eq:W_analytic}W_q(x)=\sum_{k\geq 0}q^kW_0^{*(k+1)}(x),\end{equation}
where $W_0^{*(k+1)}$ is the $(k+1)$-th convolution of $W_0$, 
see~\cite{bertoin_ergodicity, suprun} or~\cite[Sec.\ 3.3]{scale_review}.
Furthermore, the transform~\eqref{eq:transform} is still true for any $q\in\mathbb C$ and all $\theta>\Phi(|q|)$.
The following Lemma extends a basic property of the scale function to $q\in\C_+$, see Appendix for the proof. This result is not true for $q$ with $\Re(q)<0$ and, in fact, the zeros for negative $q$ have been employed in~\cite{bertoin_ergodicity}. 
\begin{lemma}\label{lem:non-zero}
It holds that $W_q(x)\neq 0$ for all $x>0$ and all $q\in \C_+$. 
\end{lemma}

According to~\eqref{eq:W_analytic} we define 
\begin{equation}\label{eq:WT}W_{-\mat T}(x)=\sum_{k\geq 0}(-\mat T)^kW_0^{*(k+1)}(x),\qquad x\geq 0\end{equation}
for any square matrix $-\mat T$. In order to claim invertibility of $W_{-\mat T}(x)$ we must assume that the eigenvalues of $-\mat T$ are in~$\C_+$. The above theory extends to the ME-time horizon in the following way.
\begin{thm}\label{thm:W}
For a square matrix $-\mat T$ with eigenvalues in $\C_+$ the matrix-valued function $W_{-\mat{T}}:[0,\infty)\mapsto \R^{p\times p}$ in~\eqref{eq:WT} is continuous, invertible for positive arguments, and is uniquely determined by the transform
\[\int_0^\infty e^{-\theta x}W_{-\mat{T}}(x)\D x=(\psi(\theta)\mat{I}+\mat{T})^{-1},\qquad \theta>\Phi(\rho_{\mat{T}}),\]
where $\rho_{\mat{T}}$ is the spectral radius of~$\mat{T}$. 

Moreover, for $T\sim \mbox{ME}(\ba,\mat T,\mat t)$ independent of $X$ there is the identity:
\[\p(\tau_x^+<\tau_{-y}^-\wedge T)=\ba W_{-\mat{T}}(y)W_{-\mat{T}}(x+y)^{-1}\mat{l},\qquad x,y\geq 0,x+y>0 .\]
\end{thm}
\begin{proof} 
Continuity in~$x$ readily follows from the power series representation of $W_{-\mat T}(x)$ and the dominated convergence theorem.
It also yields the transform expression via Fubini's theorem (cf.\ \cite[Sec.\ 3.3]{scale_review})
\begin{multline*}\sum_{k\geq 0}(-\mat{T})^k\int_0^\infty e^{-\theta x}W_0^{*(k+1)}(x)\D x=\sum_{k\geq 0}(-\mat{T})^k/\psi(\theta)^{k+1}\\
=(\mat{I} +\mat{T}/\psi(\theta))^{-1}/\psi(\theta)
=(\psi(\theta)\mat I+\mat{T})^{-1},\end{multline*}
when $\psi(\theta),\theta>0$ exceeds $\rho_{\mat{T}}$ which is equivalent to $\theta>\Phi(\rho_{\mat{T}})$.
Uniqueness is established by viewing this transform as a matrix of transforms of continuous (not necessarily positive) functions~\cite[Ch.\ II]{widder}.

Invertibility of $W_{-\mat{T}}(x),\ x>0$ is an immediate consequence of Lemma~\ref{lem:non-zero} and Lemma~\ref{lem:calculus}.
Finally, we rewrite~\eqref{eq:exit} as $\p(\tau_x^+<\tau_{-y}^-\wedge e_q)$ and apply~\eqref{eq:tool} using the matrix calculus from Lemma~\ref{lem:calculus}.
\end{proof}

The matrices $W_{-\mat{T}}(x)$ and $W_{-\mat{T}}(y)$ commute for any $x,y\geq 0$ and so the order of matrix multiplication $W_{-\mat{T}}(y)W_{-\mat{T}}(x+y)^{-1}$ is arbitrary.
Unlike the classical case, the matrix $W_{-\mat T}(x)$ does not need to have positive entries, see Section~\ref{sec:numerics}
\begin{rem}\label{rem:addTr}\rm
Theorem~\ref{thm:W} readily implies an expression for the transform $\e(e^{-\delta\tau_x^+};\tau_x^+<\tau_{-y}^-\wedge T),\delta\geq 0$ and it is only required to take the ME distribution of $T\wedge e_\delta$ for an independent exponential random variable $e_\delta$.
\end{rem}
Finally, for $T\sim\mbox{PH}(\ba,\mat T)$ the matrix valued function $W_{-\mat T}(x)$ coincides with the matrix valued scale function of the Markov additive process $(X_t,J_t)$, 
where $J$ is an independent Markov chain defining the PH distribution, see~\cite{iva_palm}. This can be easily seen by comparing the transforms.


\section{The  second scale function and further identities}\label{sec:Z}
The Skorokhod reflection of $X$ at $0$ starting from an arbitrary $x \geq  0$ is defined by
\[Y_t = X_t + R_t,\qquad R_t = -(0 \wedge \underline X_t),\]
where the non-decreasing process $R$ is often called a regulator. 
The first passage time of the reflected processes is denoted by 
\[\eta_a=\inf\{t\geq 0: Y_t>a\}.\]
Then the second scale function $Z_q(\theta,x)$, and three (closely related) basic identities are given by
\begin{align}Z_q(\theta,x) &= e^{\theta x}\Big(1-(\psi(\theta)-q)\int_0^xe^{-\theta y}W_q(y) \D y\Big)\neq 0,\nonumber\\
\e_x\Big(e^{-q\eta_a-\theta R_{\eta_a}}\Big)&=Z_q(\theta,x)/Z_q(\theta,a)\label{eq:Zratio},\\
 \e_x\Big(e^{-q\tau_0^-+\theta X_{\tau_0^-}};\tau_0^-<\tau_a^+\Big)&=Z_q(\theta,x)-W_q(x)Z_q(\theta,a)/W_q(a),\nonumber\\
\e_x\Big(e^{-q\tau_0^-+\theta X_{\tau_0^-}};\tau_0^-<\infty\Big)&=Z_q(\theta,x)-W_q(x)\frac{\psi(\theta)-q}{\theta-\Phi(q)},\nonumber
\end{align}
for  all $q,\theta,a\geq 0$ and $x\in[0,a]$, see~\cite[Thm.\ 2]{iva_palm}, \cite{avram_review}, \cite[Eq.\ (58)]{scale_review}.
The last identity for $\theta=\Phi(q)$ is interpreted in the limiting sense.

\begin{lemma}\label{lem:Z}
For any $x\geq 0$ and $\theta\in \C,\Re(\theta)\geq 0$ the function $q\mapsto Z_q(\theta,x)$ is analytic. It is non-zero for $q\in \C_+$.
\end{lemma}

This result and our usual arguments based on~\eqref{eq:tool} lead to the following extension of the above formulas.
\begin{thm}
For a square matrix $-\mat T$ with eigenvalues in $\C_+$ the matrix function
\[Z_{-\mat{T}}(\theta,x) = e^{\theta x}\Big(\mat{I}-(\psi(\theta)\mat{I}+\mat{T})\int_0^xe^{-\theta y}W_{-\mat{T}}(y) \D y\Big),\qquad x\geq 0,\Re(\theta)\geq 0\]
is continuous and invertible. Moreover, for $T\sim \mbox{ME}(\mat{\alpha},\mat{T},\mat t)$ independent of $X$, there are the identities
\begin{align*}\e_x\Big(e^{-\theta R_{\eta_a}};\eta_a<T\Big)&=\ba Z_{-\mat{T}}(\theta,x)Z_{-\mat{T}}(\theta,a)^{-1}\mat{l},\\
\e_x\Big(e^{\theta X_{\tau_0^-}};\tau_0^-<\tau_a^+\wedge T\Big)&=\ba \Big(Z_{-\mat{T}}(\theta,x)-W_{-\mat{T}}(x)Z_{-\mat{T}}(\theta,a)W_{-\mat{T}}(a)^{-1}\Big)\mat{l},\\
\e_x\Big(e^{\theta X_{\tau_0^-}};\tau_0^-<T\Big)&=\ba \Big(Z_{-\mat{T}}(\theta,x)-W_{-\mat{T}}(x)(\psi(\theta)\mat{I}+\mat{T})(\theta\mat{I}-\Phi(-\mat{T}))^{-1}\Big)\mat{l}\end{align*}
for all $\theta,a\geq 0$ and $x\in[0,a]$, and in the last identity $\theta$ is such that the inverse is well-defined. 
\end{thm}

Furthermore, we may also obtain the distribution of $X$ with reflecting/terminating barriers at an independent ME time.
For example, $\p(X_T\in \D x,T<\tau_{-a}^-\wedge\tau_b^+)$ has a density
\[u_{|-a,b|}(x)=\ba \Big(W_{-\mat{T}}(a)W_{-\mat{T}}(a + b)^{-1}W_{-\mat{T}}(b - x) - W_{-\mat{T}}(-x)\Big)\mat{l},\]
where $W_{-\mat{T}}(\cdot)$ is a zero matrix for negative arguments, see, e.g.,~\cite[Thm.\ 1]{bertoin_ergodicity} for the case of exponential killing.

\section{The Wiener-Hopf factorization}\label{sec:WH}
\subsection{The general case}
The Wiener-Hopf factorization~\cite[Thm.\ VI.5]{bertoin_book} holds for a general L\'evy process $X$, not necessarily one-sided.
It states that $\overline X_{e_q}$ is independent of $\overline X_{e_q}-X_{e_q}$ for an exponential random variable $e_q$ of rate $q>0$ independent of~$X$.
Furthermore, by time reversal the two quantities are equal in law to $X_{e_q}-\underline X_{e_q}$ and $-\underline X_{e_q}$, respectively.
For an ME-horizon we have a vector factorization:
\begin{prop}\label{prop:WH_gen}
Let $T\sim\mbox{ME}(\ba,\mat T,\mat t)$ be independent of a general L\'evy process $X$.
Then for Borel sets $A,B\subset[0,\infty)$ we have
\begin{align*}&\p(\overline X_T\in A,\overline X_T-X_T\in B)=\p(X_T-\underline X_T\in A,-\underline X_T\in B)\\
&\quad=\ba\int_0^\infty  e^{\mat Tt}\p(\overline X_t\in A)\D t\cdot \int_0^\infty e^{\mat Tt} \p(-\underline X_t\in B)\D t(-\mat T)\mat t,
\end{align*}
where both terms are finite vector-valued signed measures.
\end{prop}
\begin{proof}
Just note that
\[\int_0^\infty e^{-qt}\p(\overline X_t\in A,\overline X_t-X_t\in B)\D t=\int_0^\infty  e^{-qt}\p(\overline X_t\in A)\D t\cdot q\int_0^\infty e^{-qt} \p(-\underline X_t\in B)\D t,\]
extend this identity to $q\in \C_+$ and apply the functions on both sides to $-\mat T$ using Lemma~\ref{lem:calculus}.
By conditioning on $T$ we get the result. Countable additivity and finiteness of each term is easy to see.
\end{proof}

\subsection{The factors in spectrally-negative case}
In the case of a spectrally-negative L\'evy process $X$ both terms in Proposition~\ref{prop:WH_gen} can be written in a more explicit form.
Recall that $\overline X_{e_q}$ is exponential of rate $\Phi(q)$, whereas the law of $-\underline X_{e_q}$ is given by 
\[\p(-\underline X_{e_q}\leq x)=q\Big(\frac{1}{\Phi(q)}W_q(x)-\int_0^xW_q(y)\D y\Big),\qquad x\geq 0,\]
see~\cite[Eq.\ (8)]{bertoin_ergodicity}.
\begin{prop}\label{prop:WH}
For a square matrix $-\mat T$ with eigenvalues in $\C_+$ we have
\begin{align*}\int_0^\infty  e^{\mat Tt}\p(\overline X_t> x)\D t&=(-\mat T)^{-1}e^{-\Phi(-\mat T)x},\qquad x\geq 0,\\
\int_0^\infty  e^{\mat Tt}\p(-\underline X_t\leq x)\D t&=\Phi(-\mat T)^{-1}W_{-\mat T}(x)-\int_0^xW_{-\mat T}(y)\D y,\qquad x\geq 0.
\end{align*}
\end{prop}
\begin{proof}
The first statement follows by applying $q\int_0^\infty e^{-q t}\p(\overline X_t>x)\D t=e^{-\Phi(q)x}$ to $-\mat T$. 
We get the second from the above representation of $\p(-\underline X_{e_q}\leq x)$. Recall that the function $q\mapsto\int_0^x W_q(y)\D y$ has been discussed in Section~\ref{sec:Z}.
\end{proof}
Note that the above expression can be written in terms of $Z_{-\mat T}$ using the obvious identity
\[\int_0^xW_{-\mat T}(y)\D y=(Z_{-\mat T}(0,x)-\mat I)(-\mat T)^{-1}.\]
Let us mention that the known bivariate transform~\cite[Thm.\ VII.4]{bertoin_book} can also be easily extended:
\[\e e^{-u\overline X_T-v(\overline X_T-X_T)}=\ba\big(u\mat{I}+\Phi(-\mat{T})\big)^{-1}\big(v\mat{I}-\Phi(-\mat{T})\big)\big(\psi(v)\mat{I}+\mat{T}\big)^{-1} \mat{t}\]
for $u,v\geq 0$ such that the latter inverse is well-defined.

\subsection{On the derivative of the scale function}
It is well known that $W_q(x),q\geq 0$ is differentiable in $x$ apart from countably many points with the right derivative well-defined for all $x\geq 0$. Furthermore, $W_q(0)=c$ where 
\begin{equation}\label{eq:c}
c=1/d\text{ if }X\text{ is a b.v.\ process with linear drift }d>0,\text{ and }c=0\text{ otherwise}.
\end{equation}
By convention we choose the right derivative $W'_q(x)=\partial_+ W_{q}(x)$ and note that $W_q(\D y)=W'_q(y)\D y+c\delta_0(\D y),y\geq 0$.
The following basic result is proven in Appendix.
\begin{lemma}\label{lem:Wdiff}
For any $q\in\C,x\geq 0$ it holds that 
\begin{align*}
W'_q(x)=\partial_+ \sum_{k\geq 0} q^k W_0^{*(k+1)}(x)=\sum_{k\geq 0} q^k \partial_+ W_0^{*(k+1)}(x),
\end{align*}
and the series is absolutely convergent.
\end{lemma}
Thus we may define $W'_{-\mat T}(x)\in\R^{p\times p}$ and using the bounds in the proof of Lemma~\ref{lem:Wdiff} we readily find that
\[W'_{-\mat{T}}(x)=\sum_{k\geq 0} (-\mat T)^k \partial_+ W_0^{*(k+1)}(x)=\partial_+ \Big(W_{-\mat{T}}(x)\Big).\]
Finally, we obtain
\[W_{-\mat T}(x)=c\mat I+\int_0^x W'_{-\mat T}(y)\D y,\qquad x\geq 0,\]
which in combination with Proposition~\ref{prop:WH_gen} and Proposition~\ref{prop:WH} gives
\begin{multline}\label{eq:WHdens}\p(\overline X_T\in \D x,\overline X_T-X_T\in \D y)=\p(X_T-\underline X_T\in \D x,-\underline X_T\in \D y)\\=\ba e^{-\Phi(-\mat{T})x}\D x\times\Big(c\mat{I}\delta_0(\D y)+(W'_{-\mat{T}}(y)-\Phi(-\mat{T})W_{-\mat{T}}(y))\D y\Big)\mat{t}\end{multline}
for $x,y\geq 0$ and $c$ defined in~\eqref{eq:c}.

\subsection{An application to financial mathematics}
Exponentials of L\'evy processes are used extensively in financial mathematics to model stock prices.
They also appear in so--called unit--linked (or equity linked) life insurance models where an option is exercised upon death of an insured. 
In this context it is important that one can specify the distribution of the time horizon $T$ in a general and tractable way. 
In \cite{Gerber-Shiu-2013} the time horizon $T$ is assumed to have a density which is a linear combination (not necessarily being a mixture) of exponential densities, the class of which is dense in the class of distributions on the positive reals. 
In this way the classical Wiener--Hopf decomposition can be employed for each of the terms in the combination. 
They further assume that jumps (in both directions) are also combinations of exponentials, leading to a rational L\'evy exponent. 
Below we provide a respective formula for an arbitrary ME time horizon $T$ and a spectrally-negative L\'evy process.


In financial mathematics context, the following identity is well known:
\[
e^u\e\big(e^{-u}-e^{\underline X_{e_q}}\big)^+=Z_q(0,u)-\frac{q}{\Phi(q)}\frac{\Phi(q)-1}{q-\psi(1)}Z_q(1,u),\qquad q>0,u\geq 0.
\]
see~\cite[Cor.\ 9.3]{kyprianou} and references therein. One may also derive this identity in a few different ways using the above presented formulas in the classical setting of~$q>0$.
Now we can extend this result to ME time horizon, where we also include the second Wiener-Hopf factor.
The following result can be proven using~\eqref{eq:WHdens} and other related ME formulas, but it is much easier to simply extend the desired formula from exponential to ME time.

\begin{lemma}\label{lem:finance}
Let $\beta\in\R$ and consider $T\sim\mbox{ME}(\ba,\mat T,\mat t)$ independent of $X$, such that the real parts of the eigenvalues of $-\mat T$ exceed~$\psi(\beta)$ when $\beta>\Phi(0)$.
Then for $u\geq 0$ we have
\begin{multline}
e^{u}\e\Big(\big(e^{-u}-e^{\underline X_T}\big)^+e^{\beta(X_T-\underline X_T)}\Big)
=\ba\Big(\Phi(-\mat T)(\Phi(-\mat T)-\beta\mat I)^{-1}Z_{-\mat T}(0,u)\\
+(-\mat T)(\Phi(-\mat T)-\mat I)(\Phi(-\mat T)-\beta\mat I)^{-1}(\mat T+\psi(1)\mat I)^{-1}Z_{-\mat T}(1,u)\Big)\mat l,
\end{multline}
assuming that $\psi(1)$ is not an eigenvalue of $-\mat T$.
\end{lemma}
\begin{proof}
Recall that $\e \exp\left(\beta(X_{e_q}-\underline X_{e_q})\right)=\Phi(q)/(\Phi(q)-\beta)$ assuming $\Phi(q)>\beta$.
If $\beta\leq \Phi(0)$ then this condition is always satisfied, and otherwise it is equivalent to $q>\psi(\beta)$.
Now for $T=e_q$ we have
\begin{align*}
e^{u}\e\Big(\big(e^{-u}-e^{\underline X_T}\big)^+e^{\beta(X_T-\underline X_T)}\Big)(\psi(1)-q)=(\psi(1)-q)\frac{\Phi(q)}{\Phi(q)-\beta}Z_{q}(0,u)+q\frac{\Phi(q)-1}{\Phi(q)-\beta}Z_{q}(1,u).
\end{align*}
In the case $\beta\leq \Phi(0)$ this identity readily extends to $q\in\C_+$ and the result follows from~\eqref{eq:tool}.
Otherwise, the formula is true for $q\in\C_+$ such that $\Re(q)>\psi(\beta)$ and we may use~\eqref{eq:tool} by assumption on the eigenvalues of~$-\mat T$.
Finally, if $\Phi(-\mat T)$ had an eigenvalue $\beta>\Phi(0)$ then $\psi(\beta)$ would be an eigenvalue of $-\mat T$ according to $\psi(\Phi(-\mat T))=-\mat T$.
Thus $\Phi(-\mat T)-\mat I$ is invertible.
\end{proof}

\section{Examples}
\subsection{Strictly stable L\'evy processes}\label{sec:stable}
A strictly stable (or self-similar) L\'evy process, $(X_t)_{t\geq 0}$, is characterized by the property that $(X_{ct})_{t\geq 0}$ has the law of $(c^{1/\a}X_t)_{t\geq 0}$ for any $c>0$ and some~$\a>0$, see~\cite[Ch.\ 3]{sato}.
In the spectrally-negative case it must be that $\a\in(1,2]$, where $\a=2$ corresponds to a Brownian motion.
If we (without real loss of generality) fix the scale parameter so that $\psi(1)=1$, then we readily get
\begin{equation}\label{eq:stable}\psi(\theta)=\theta^{\a},\qquad\a\in(1,2],\ \theta\geq 0.\end{equation}
Such processes have received a great deal of attention in the literature, see, for example,~\cite{bertoin_stable,peskir,Coqueret}.

Here we derive explicit formulas for the basic objects appearing in the above fluctuation identities. 
Firstly, $\Phi(q)=q^{1/\a},q\geq 0$ and thus
\[\Phi(-\mat{T})=(-\mat T)^{1/\a}.\]
The other formulas are based on the 
Mittag-Leffler function:
\[E_{\a,\beta}(z)=\sum_{n=0}^\infty\frac{z^n}{\Gamma(\a n+\beta)},\qquad \alpha,\beta>0,\]
which is analytic for $z\in \C$.
The scale function $W_q(x)$ is given in terms of the derivative of the Mittag-Leffler function~\cite{bertoin_ergodicity}:
$W_q(x)=\a x^{\a-1}E'_{\a,1}(qx^\a),\quad q,x\geq 0$. 
Noting that
\[  E_{\a,1}^\prime (z)=\sum_{n=1}^\infty\frac{nz^{n-1}}{\Gamma(1+\a n)} = \sum_{n=1}^\infty \frac{n z^{n-1}}{\a n \Gamma(\a n)}  = \alpha^{-1} E_{\a,\a}(z), \]
we arrive at the formula
\[W_{-\mat{T}}(x)=\a x^{\a-1}E'_{\a,1}(-x^\a\mat{T})= x^{\alpha-1} E_{\a,\a} 
\left( -x^{\a}\mat{T}  \right) 
, \qquad x\geq 0.\]
From the well known identity~\cite[Sec. 4.3]{gorenflo2014mittag}
we then obtain
\[  W^\prime_{-\mat{T}}(x)=x^{\alpha-2}  E_{\alpha,\alpha-1}(-x^\alpha\mat{T})  \]
Results concerning the second scale function are formulated in the following lemma.
\begin{lemma}
In the stable case there are the identities for $x\geq 0$
\[Z_{-\mat{T}}(0,x)=\Big(\mat{I}+E_{\a,1}(-\mat{T}x^{\a})\Big),\qquad Z_{-\mat{T}}(\theta,x)=\theta^{-\a}\sum_{n=0}^\infty (-\theta^{-\a}\mat{T})^n 
\frac{\gamma (x\theta; \alpha n + \alpha)}{\Gamma (\alpha n +\alpha)},
 \]
 where $\gamma (y;\beta)=\int_0^y x^{\beta-1}e^{-z}\D z$ denotes the lower incomplete Gamma function.
\end{lemma}
\begin{proof}
We readily find that $\int_0^x W_q(y)\D y=E_{\a,1}(qx^{\a})/q$ leading to the first formula.
For the second identity, simply write $E_{\a,\a}(-x^{\alpha}\mat{T})$ in terms of its series expansion and interchange summation and integration, which can be justified by a dominated convergence argument. 
\end{proof}

\subsection{The scale function in the case of ME jumps}\label{sec:MEjumps}
Here we assume that the generator matrix is diagonalizable, see~\eqref{eq:diagonalizable}, which readily gives
\[W_{-\mat T}(x)=\mat P\diag_k\big(W_{\lambda_k}(x)\big)\mat P^{-1}.\]
Thus it is sufficient to obtain $W_q(x)$ for $q\in\C_+$.

Assume that $X$ has finite jump activity and the jump distribution is ME. In other words we consider a Cram\'er-Lundberg process with ME jumps possibly perturbed by an independent Brownian motion.
The classical ($q\geq 0$) scale function is explicit in this case, see~\cite{yamazaki,scale_review}.
It is given in terms of the zeros of the rational function $\psi(\theta)-q$. 
Importantly, this result readily extends to any $q\in\C$:
\[
W_q(x)=\sum_{z\in\C:\;\psi(z)=q}\frac{1}{\psi'(z)}e^{z x},\qquad x\geq 0,
\]
assuming that the roots are distinct. Indeed, its transform is $\sum \frac{1}{\psi'(z_i)(\theta-z_i)}$ which is a partial fraction decomposition of $1/(\psi(\theta)-q)$.
The case of non-distinct roots, which is a bit more cumbersome, is also analogous to the case $q\geq 0$.

\section{Numerics}\label{sec:numerics}
For the purpose of illustration we consider 
a strictly stable  L\'evy process $X$ with $\psi(\theta)=\theta^\a$ and $\a=1.5$ as discussed in Section~\ref{sec:stable}, and an ME density
\[  f_T(x)=\frac{17}{9}e^{-x}\cos^2(2x),\qquad x\geq 0 .  \]
This is a genuine matrix exponential distribution which is not of phase--type. 
A (canonical) representation is given by 
\[   \ba= \left( -\frac{8}{9},-\frac{34}{9},\frac{17}{3}   \right), \ \ \mat{T} =
\begin{pmatrix}
0 &-17 & 17 \\
3 & 2 & -6 \\
2 & 2 & -5 
\end{pmatrix},
  \]
and the eigenvalues for $\mat{T}$ are $-1,-1\pm 4i$. 
Some paths of $X$ and the ME density $f_T$ are depicted in Figure~\ref{fig:paths}.
  \begin{figure}[h!]
  \includegraphics[width=0.47\textwidth]{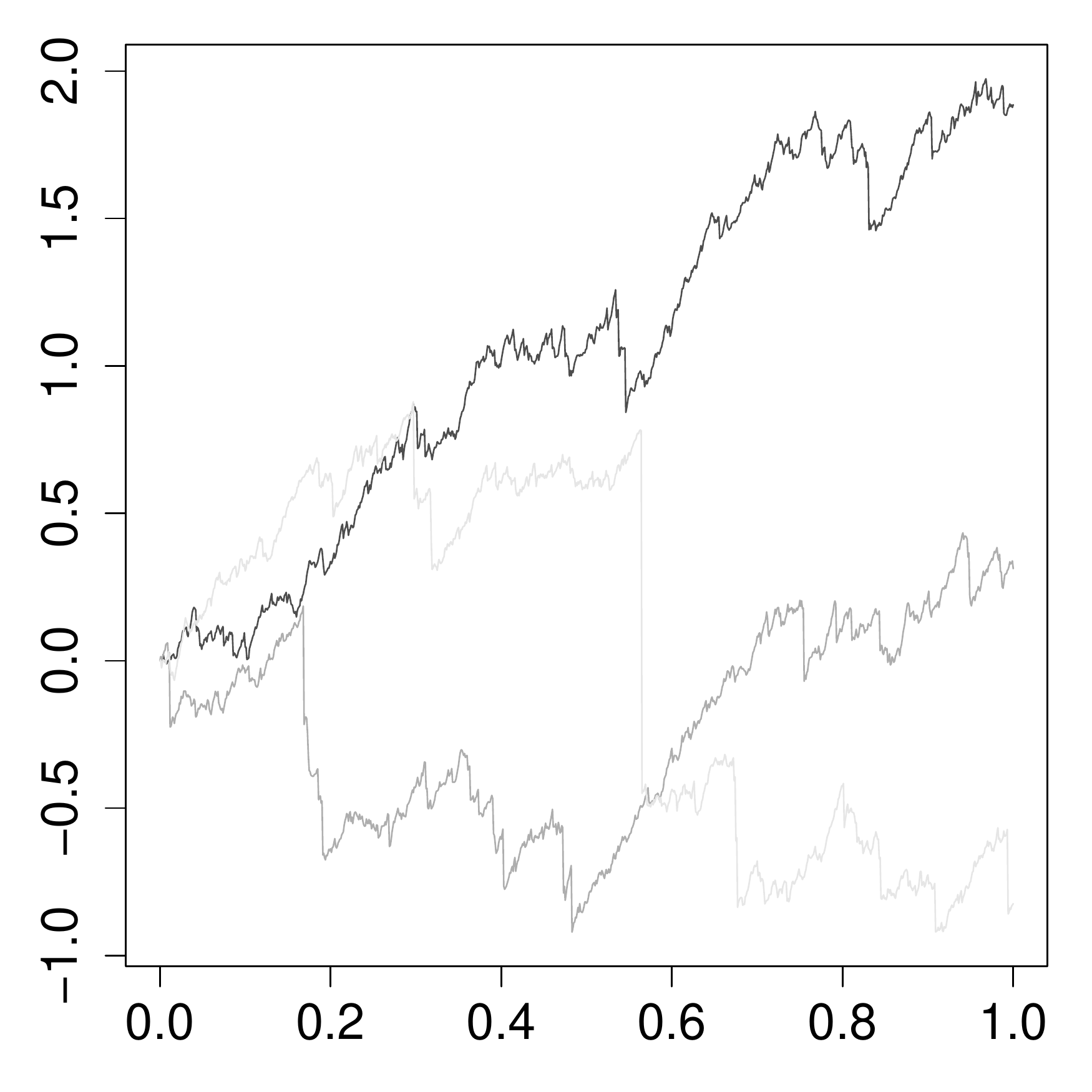}\quad
  \includegraphics[width=0.47\textwidth]{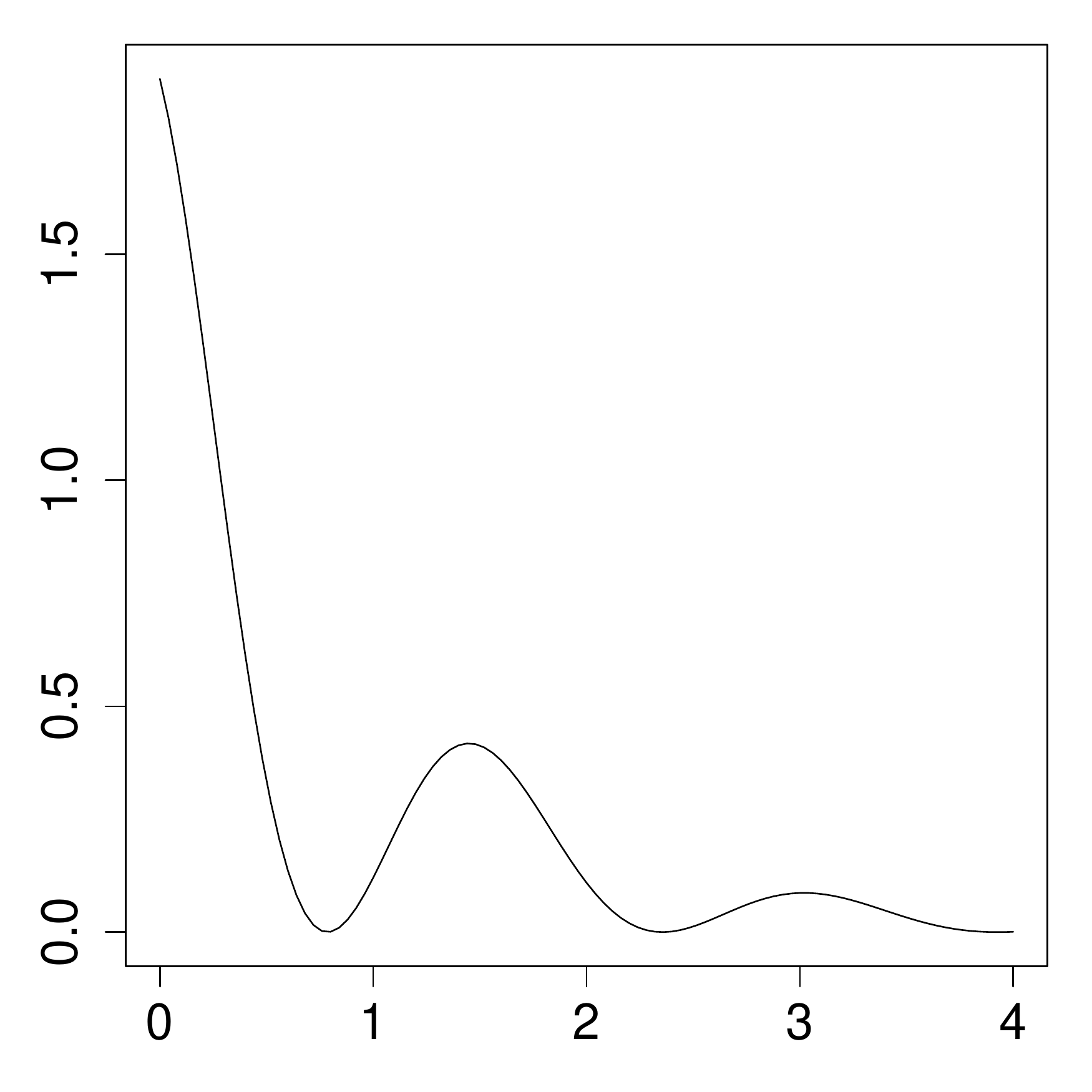}
  \caption{Sample paths of $X$ and the ME density $f_T(x)$.}
  \label{fig:paths}
  \end{figure}

Our numerical experiments are carried out in \textsc{R} language. All the implemented formulas are checked against simulations, where the process $X$ is sampled on the grid with step size $0.001$.
The time horizon $T$ is simulated by inversion which uses a numerical root finding routine.
The number of sampled paths is set to $3000$. Note that there is both a systematic error due to sampling on the grid and a Monte Carlo error.

Firstly, we compute 
\[\Phi(-\mat T)=(-\mat T)^{1/\a}=
\begin{pmatrix}
1.13 & 9.79 & -10.46 \\
-1.49 & 1.64  & 0.74 \\
 -0.99 & 0.42 &  1.49
\end{pmatrix}.
  \] 
Secondly, the scale function $W_{-\mat T}(x)$ is implemented using diagonalization and the standard  Mittag-Leffler function, see Section~\ref{sec:stable}.
We provide its entries in Figure~\ref{fig:W}.
Observe that these need not be positive or monotone. Note also that one may expect numerical issues for larger argument values, and so it may be useful to rewrite certain formulas in terms of a normalized version $e^{-\Phi(-\mat T)x}W_{-\mat T}(x)$.
  \begin{figure}[h!]
  \includegraphics[width=0.47\textwidth]{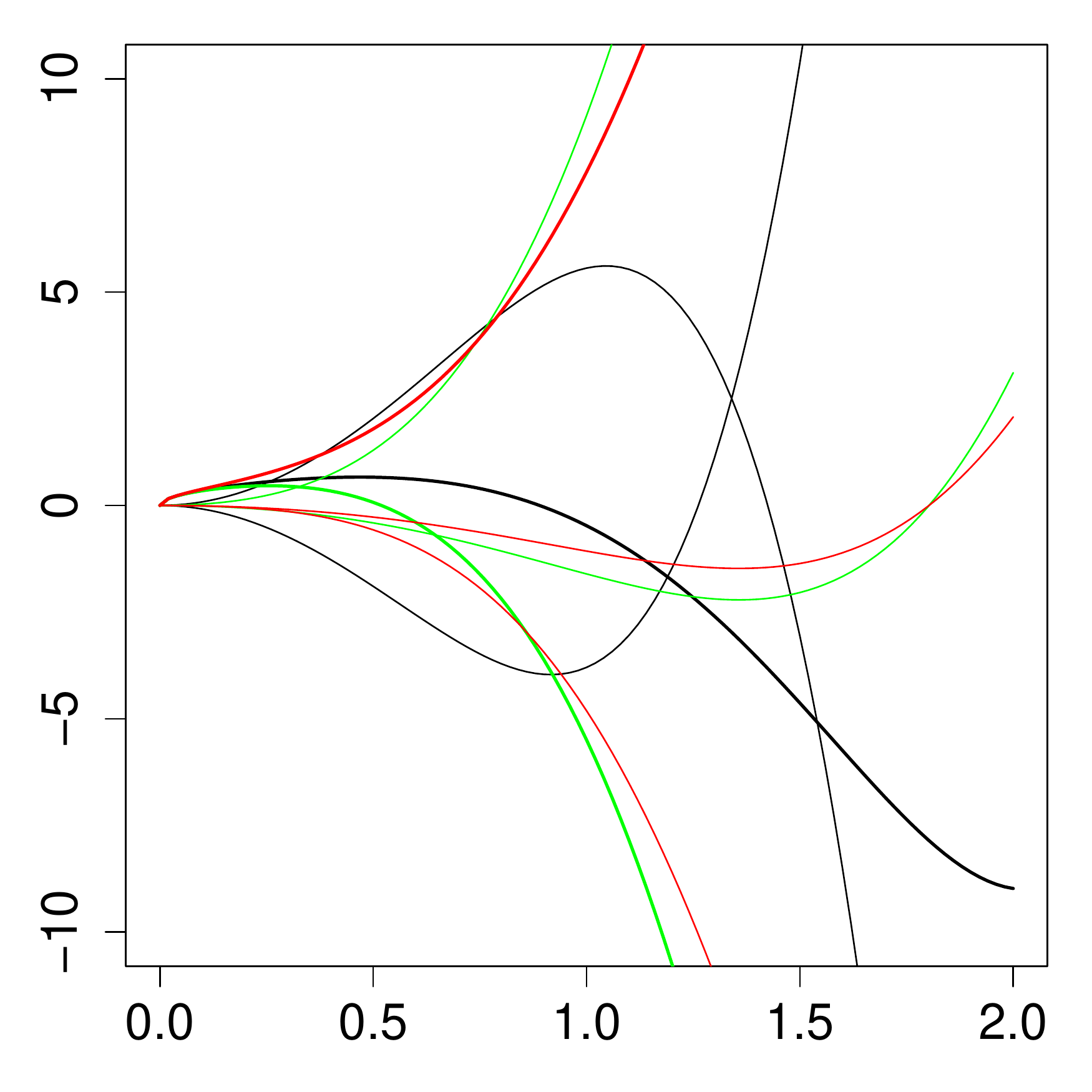}\quad
    \includegraphics[width=0.47\textwidth]{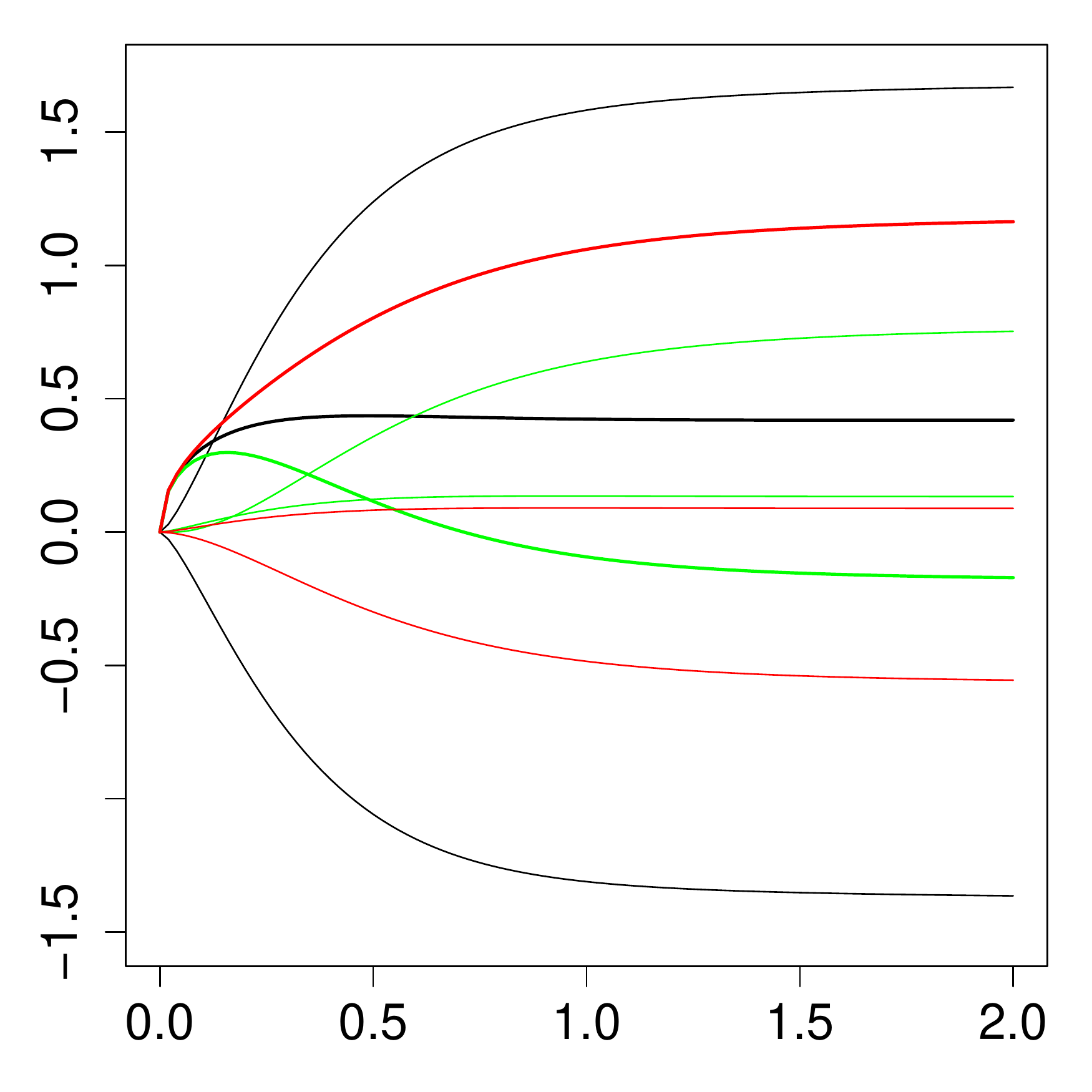}
  \caption{The entries of the matrix-valued scale function $W_{-\mat T}(x)$ and its modified version $e^{-\Phi(-\mat T)x}W_{-\mat T}(x)$: colors correspond to rows and the diagonal elements are in bold.}
  \label{fig:W}
  \end{figure}
  
 Having the basic constituents $\Phi(-\mat T)$ and $W_{-\mat T}(x)$, we may now illustrate our formulas
The one- and two-sided first passage probabilities, $\p(\tau_x^+<T)$ and $\p(\tau_x^+<\tau_{x-1}^-\wedge T)$, 
obtained using Theorem~\ref{thm:Phi} and Theorem~\ref{thm:W} are presented in Figure~\ref{fig:passage} (Left).
For comparison, we also provide simulated values of these probabilities, and note that these are rather close to the exact values obtained from our formulas. 
  \begin{figure}[h!]
  \includegraphics[width=0.47\textwidth]{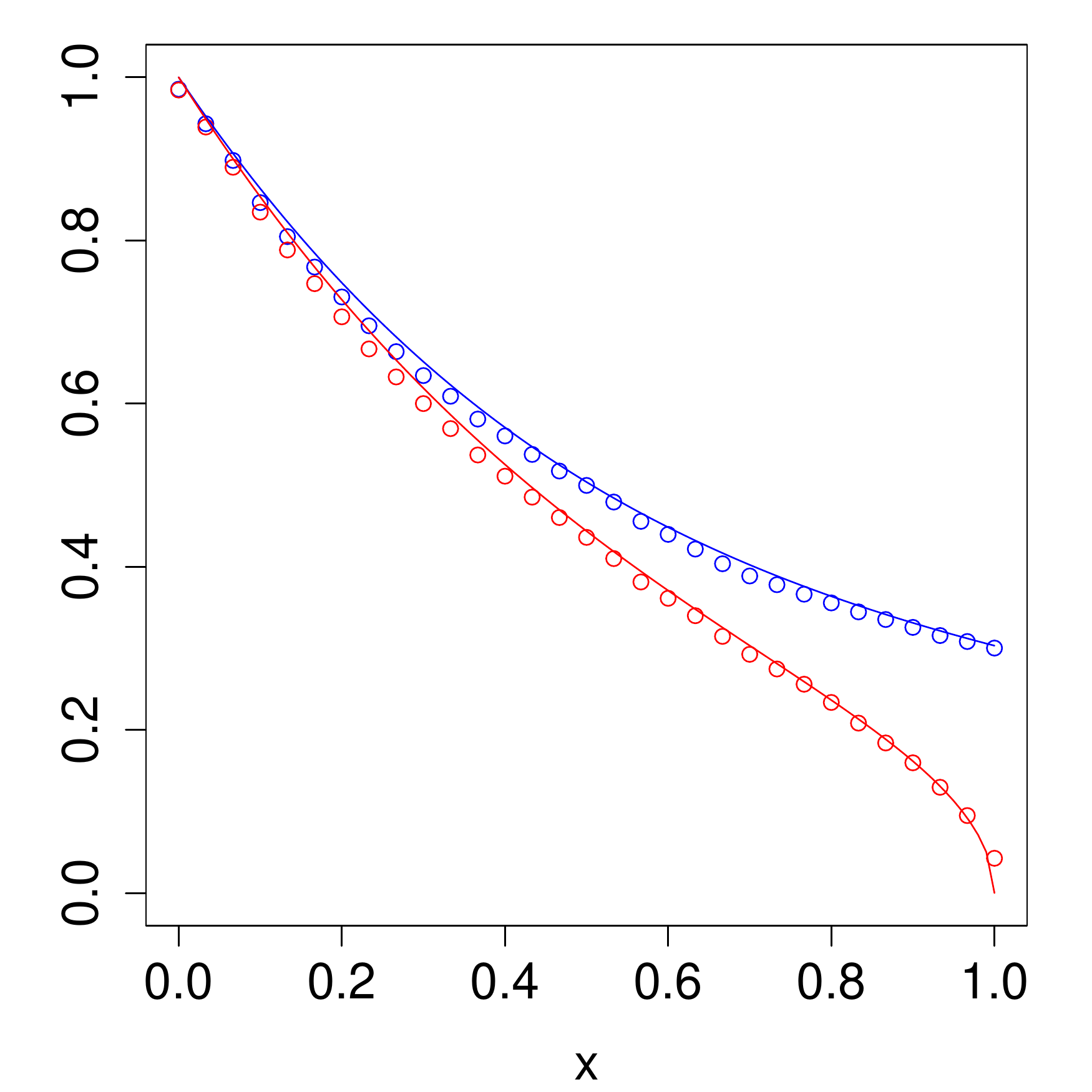}
    \includegraphics[width=0.47\textwidth]{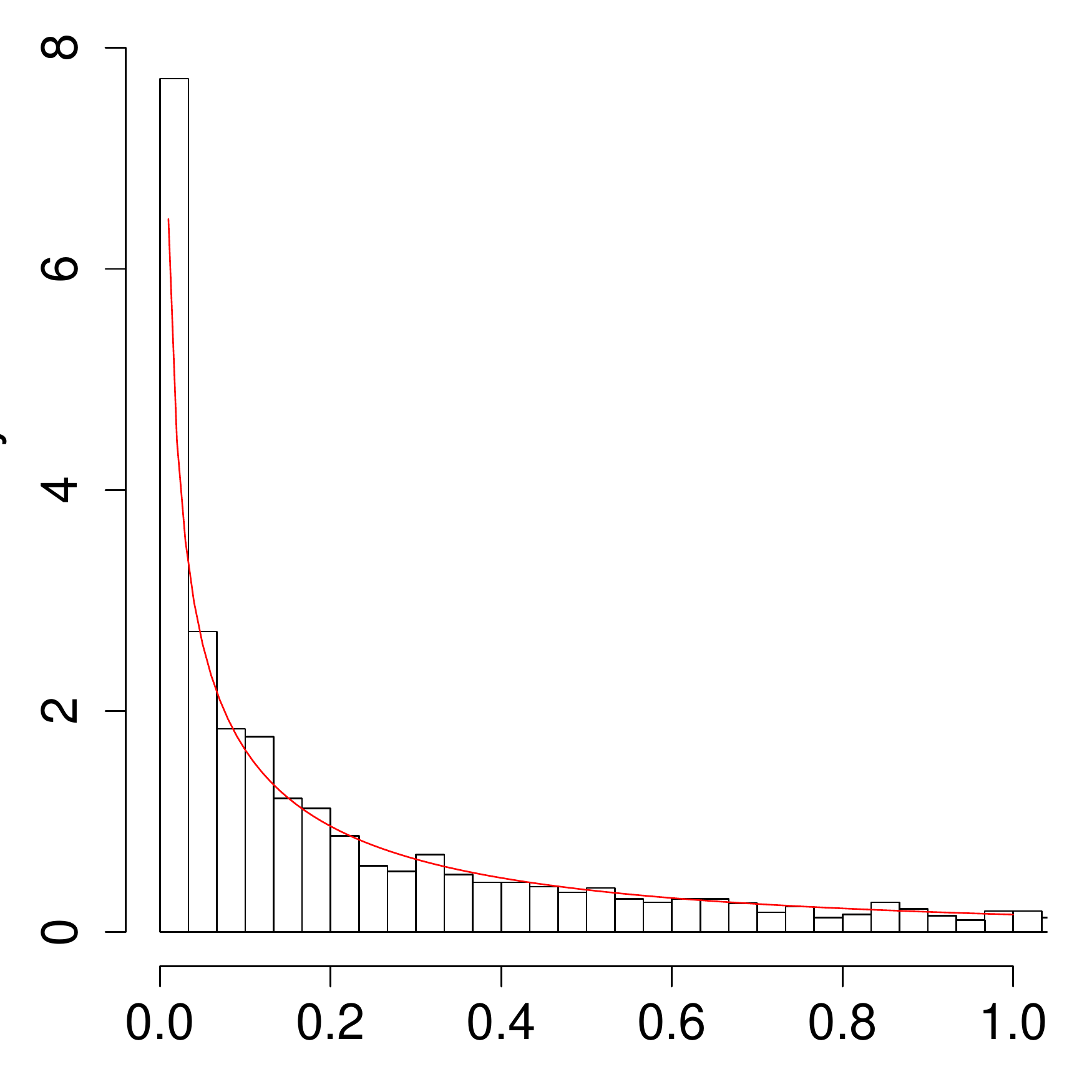}
  \caption{Left: probabilities $\p(\tau_x^+<T)$ (blue) and $\p(\tau_x^+<\tau_{x-1}^-\wedge T)$ (red) for $x\in[0,1]$; simulated values in dots. Right: histogram of $-\underline X_T$ and its density in red; truncated to the interval $[0,1]$.}
  \label{fig:passage}
  \end{figure}
  
Next, we consider the density of $-\underline X_T$ which can be easily obtained from~\eqref{eq:WHdens}.
Its expression involves $W'_{-\mat T}(x)$ and we refer to Section~\ref{sec:stable} for the corresponding formula.
In Figure~\ref{fig:hist} (Right) we plot this density over a histogram of simulated values of~$-\underline X_T$.
Again we observe an agreement between the exact formula and simulations.

\section{ME inter-observation times: an open problem}
An important application of the fluctuation theory for an exponential time horizon concerns risk processes observed at the epochs of an independent Poisson process, see~\cite{alb_iva_zhou} a references therein for a suite of identities.
Generalization of these identities to the case when observation epochs form an independent renewal  process with ME inter-arrival times is not straightforward, and requires further study.
In the case of PH inter-arrivals one may resort to the existing theory for Markov additive process, which we now demonstrate. For an alternative approach in the case of Erlang interarrival times see~\cite{albrecher2013randomized}.

Let $(\widehat T_i)_{i\geq 1}$ be the epochs of an independent renewal process with $\mbox{PH}(\ba,\mat T)$ inter-arrival times.
Define the ruin time as
$\widehat \tau_0=\inf\{\widehat T_i:X_{\widehat T_i}<0\}$, which corresponds to the first observation of the process while it is negative.
A slight adaptation of~\cite[Thm.\ 4.1]{albrecher2013risk} and its proof to the present setting yields the following result.
\begin{prop}\label{prop:interobservations}
For any $x\geq 0$ and $\mbox{PH}(\ba,\mat T)$ interarrival distribution  there is the formula
\[\p(\tau_x^+<\widehat \tau_0)=\ba e^{-\Phi(-\mat T) x}\Big(\mat I-\int_0^x W_{-\mat T-\mat t\ba}(y)\mat t\ba e^{-\Phi(-\mat T)y}\D y\Big)^{-1}\mat 1.\]
\end{prop}
Note that $-\Phi(-\mat T)$ is the transition rate matrix of the first passage Markov chain for the model with termination according to the given PH distribution, see Corollary~\ref{cor:G}.
Furthermore, $\mat T+\mat t\ba$ is the transition rate matrix of the Markov chain representing the PH distribution which is restarted according to $\ba$ upon termination.
It corresponds to a recurrent Markov chain, and without real loss of generality we may assume it is irreducible. 
Thus $\mat T+\mat t\ba$ has a single eigenvalue at~0. Nevertheless $W_{-\mat T-\mat t\ba}(y)$ coincides with the scale matrix of our Markov additive process, which can be seen by comparing the transforms.

An extension of Proposition~\ref{prop:interobservations} to an arbitrary positive initial capital can be obtained from the strong Markov property of the associated Markov additive process. Furthermore, the limit $\p(\widehat \tau_0=\infty)$ obtained by $x\to\infty$ has an explicit form in terms of a solution to a certain Sylvester equation, which is derived analogously to~\cite[Thm.\ 4.2]{albrecher2013risk}.
It is an interesting open question on how to generalize the above result to ME inter-arrival times where a probabilistic interpretation is not readily available.

\appendix
\section{Remaining proofs}
\begin{proof}[Proof of Lemma~\ref{lem:Phipsi}]
The first statement is well known, and it readily follows from~\eqref{eq:taux}.
For any $x>0$ a standard application of optional stopping yields
\[\e e^{\theta X_{\tau_x^+\wedge t}-\psi(\theta)\tau_x^+\wedge t}=1,\qquad t,\theta>0.\]
Noting that $X_{\tau_x^+\wedge t}\leq x$ we may extend this identity to all $\theta\in \C_+$. Next, we assume that $\psi(\theta)\in \C_+$ and let $t\to\infty$.
Bu the dominated convergence we readily obtain
\[\e(e^{\theta x-\psi(\theta)\tau_x^+};\tau_x^+<\infty)=\e e^{\theta x-\psi(\theta)\tau_x^+}=1,\]
which is a standard identity but now extended to $\theta\in\C_+$ such that $\psi(\theta)\in\C_+$.
But now 
\[e^{-\theta x}=\e e^{-\psi(\theta)\tau_x^+}=e^{-\Phi(\psi(\theta))x}\]
according to~\ref{eq:taux} continued to~$\C_+$. Thus indeed $\psi(\Phi(\theta))=\theta$ and the uniqueness result is obvious.
\end{proof}

\begin{proof}[Proof of Lemma~\ref{lem:non-zero}]
For any $q,p\in\C$ and $x>0$ there is the identity (see~\cite[Eg.\ (6)]{loeffen14} for $q,p\geq 0$):
\[(q-p)\int_0^xW_p(x-y)W_q(y)\D y=W_q(x)-W_p(x),\]
which readily follows by taking transforms of both sides in $x>0$ and using~\eqref{eq:transform}.
Here we apply Fubini on the left hand side and in that respect we note the bound $|W_q(x)|\leq W_{|q|}(x)$ which readily follows from~\eqref{eq:W_analytic}.
Moreover,  for $y\in[0,x]$ and $q\in \C_+$ we also have
\[\e(e^{-q\tau_{x-y}^+};\tau_{x-y}^+<\tau_{-y}^-)W_q(x)=W_q(y)\]
obtained by analytic continuation.
Thus if $W_q(x)=0$ then $W_q(y)=0$ for all $y\in[0,x]$, but then from the above formula we get $W_p(x)=0$ for all $p\in \C$, contradicting the basic fact that $W_p(x)>0$ for $p>0$.
\end{proof}

\begin{proof}[Proof of Lemma~\ref{lem:Z}]
Use~\eqref{eq:W_analytic} to observe that
\[\int_0^xe^{-\theta y}W_q(y) \D y=\sum_{k\geq 0} q^k\int_0^xe^{-\theta y}W_0^{*(k+1)}(y)\D y,\qquad q\in \C,\]
where the series is absolutely convergent. Thus $Z_q(\theta,x)$ can be analytically continued to $q\in \C$.

The proof of $Z_q(\theta,x)\neq 0$ for $q\in\C_+$ follows the arguments in the proof of Lemma~\ref{lem:non-zero}.
Here we rely upon~\eqref{eq:Zratio} and the identity
\[(p - q)\int_0^aW_p(a-x)Z_q(x, \theta) \D x = Z_p(a,\theta) - Z_q(a, \theta)\]
extending~\cite[Lem.\ 4.1]{alb_iva_zhou} to all $p,q\in \C$.
\end{proof}


\begin{proof}[Proof of Lemma~\ref{lem:Wdiff}]
Recall that $W^{*k}$ is the $k$th convolution power of the non-decreasing, non-negative scale function $W(x)=W_0(x)$; here we drop the subscript to simplify notation.
We start by observing for any $h>0$ that 
\begin{align}&\frac{W^{*(k+1)}(x+h)-W^{*(k+1)}(x)}{h}\label{eq:Wder}\\
&\quad=\frac{1}{h}\int_x^{x+h}W^{*k}(y)W(x+h-y)\D y+\frac{1}{h}\int_0^{x}W^{*k}(y)(W(x+h-y)-W(x-y))\D y.\nonumber\end{align}
Note that $W^{*k}(x),W(x)$ are non-decreasing functions of $x$ and so the first term is upper bounded by $W^{*k}(x+h)W(h)$.
The second term is upper bounded by
\[W^{*k}(x)\frac{1}{h}\int_0^{x}(W(x+h-y)-W(x-y))\D y\leq W^{*k}(x)\frac{1}{h}\int_x^{x+h}W(y)\D y\leq W^{*k}(x)W(x+h),\]
where the first inequality is obtained by dropping out the term $-\int_0^hW(y)\D y$.
But
\[\sum_{k\geq 0} |q|^k (W^{*k}(x+h)W(h)+W^{*k}(x)W(x+h))<C<\infty\] 
for all $h\in(0,1]$ and hence the dominated convergence theorem can be applied.

\end{proof}

\section*{Acknowledgments}
The second author gratefully acknowledges the financial support of Sapere Aude Starting Grant 8049-00021B ``Distributional Robustness in Assessment of Extreme Risk''.


\begin{thebibliography}{10}

\bibitem{albrecher2013randomized}
H.~Albrecher, E.~C. Cheung, and S.~Thonhauser.
\newblock Randomized observation periods for the compound Poisson risk model:
  the discounted penalty function.
\newblock {\em Scandinavian Actuarial Journal}, 2013(6):424--452, 2013.

\bibitem{albrecher2013risk}
H.~Albrecher and J.~Ivanovs.
\newblock A risk model with an observer in a Markov environment.
\newblock {\em Risks}, 1(3):148--161, 2013.

\bibitem{alb_iva_zhou}
H.~Albrecher, J.~Ivanovs, and X.~Zhou.
\newblock Exit identities for {L}\'{e}vy processes observed at {P}oisson
  arrival times.
\newblock {\em Bernoulli}, 22(3):1364--1382, 2016.

\bibitem{usabel}
S.~Asmussen, F.~Avram, and M.~Usabel.
\newblock Erlangian approximations for finite-horizon ruin probabilities.
\newblock {\em Astin Bull.}, 32(2):267--281, 2002.

\bibitem{WH_PH}
S.~Asmussen and J.~Ivanovs.
\newblock A factorization of a {L}\'{e}vy process over a phase-type horizon.
\newblock {\em Stoch. Models}, 34(4):397--408, 2018.

\bibitem{avram_review}
F.~Avram, D.~Grahovac, and C.~Vardar-Acar.
\newblock The {$W$}, {$Z$} scale functions kit for first passage problems of
  spectrally negative {L}\'{e}vy processes, and applications to control
  problems.
\newblock {\em ESAIM Probab. Stat.}, 24:454--525, 2020.

\bibitem{berg_1993}
C.~Berg, K.~Boyadzhiev, and R.~Delaubenfels.
\newblock Generation of generators of holomorphic semigroups.
\newblock {\em Journal of the Australian Mathematical Society. Series A. Pure
  Mathematics and Statistics}, 55(2):246–269, 1993.

\bibitem{peskir}
V.~Bernyk, R.~C. Dalang, and G.~Peskir.
\newblock The law of the supremum of a stable {L}\'{e}vy process with no
  negative jumps.
\newblock {\em Ann. Probab.}, 36(5):1777--1789, 2008.

\bibitem{bertoin_book}
J.~Bertoin.
\newblock {\em L\'{e}vy processes}, volume 121 of {\em Cambridge Tracts in
  Mathematics}.
\newblock Cambridge University Press, Cambridge, 1996.

\bibitem{bertoin_stable}
J.~Bertoin.
\newblock On the first exit time of a completely asymmetric stable process from
  a finite interval.
\newblock {\em Bull. London Math. Soc.}, 28(5):514--520, 1996.

\bibitem{bertoin_ergodicity}
J.~Bertoin.
\newblock Exponential decay and ergodicity of completely asymmetric {L}\'{e}vy
  processes in a finite interval.
\newblock {\em Ann. Appl. Probab.}, 7(1):156--169, 1997.

\bibitem{bladt_nielsen_book}
M.~Bladt and B.~F. Nielsen.
\newblock {\em Matrix-exponential distributions in applied probability},
  volume~81 of {\em Probability Theory and Stochastic Modelling}.
\newblock Springer, New York, 2017.

\bibitem{Coqueret}
G.~Coqueret.
\newblock On the supremum of the spectrally negative stable process with drift.
\newblock {\em Statist. Probab. Lett.}, 107:333--340, 2015.

\bibitem{jordan}
B.~D'Auria, J.~Ivanovs, O.~Kella, and M.~Mandjes.
\newblock First passage of a {M}arkov additive process and generalized {J}ordan
  chains.
\newblock {\em J. Appl. Probab.}, 47(4):1048--1057, 2010.

\bibitem{yamazaki}
M.~Egami and K.~Yamazaki.
\newblock Phase-type fitting of scale functions for spectrally negative
  {L}\'{e}vy processes.
\newblock {\em J. Comput. Appl. Math.}, 264:1--22, 2014.

\bibitem{Gerber-Shiu-2013}
H.~U. Gerber, E.~S. Shiu, and H.~Yang.
\newblock Valuing equity-linked death benefits in jump diffusion models.
\newblock {\em Insurance: Mathematics and Economics}, 53(3):615 -- 623, 2013.

\bibitem{gorenflo2014mittag}
R.~Gorenflo, A.~A. Kilbas, F.~Mainardi, S.~V. Rogosin, et~al.
\newblock {\em Mittag-Leffler functions, related topics and applications},
  volume~2.
\newblock Springer, 2014.

\bibitem{matrixFunctions}
N.~J. Higham.
\newblock {\em Functions of matrices}.
\newblock Society for Industrial and Applied Mathematics (SIAM), Philadelphia,
  PA, 2008.
\newblock Theory and computation.

\bibitem{horvath2020numerical}
G.~Horv{\'a}th, I.~Horv{\'a}th, S.~A.-D. Almousa, and M.~Telek.
\newblock Numerical inverse Laplace transformation using concentrated matrix
  exponential distributions.
\newblock {\em Performance Evaluation}, 137:102067, 2020.

\bibitem{miklos}
G.~Horv\'{a}th, I.~Horv\'{a}th, and M.~Telek.
\newblock High order concentrated matrix-exponential distributions.
\newblock {\em Stoch. Models}, 36(2):176--192, 2020.

\bibitem{thesis}
J.~Ivanovs.
\newblock {\em One-sided {M}arkov additive processes and related exit
  problems}.
\newblock BOXPress, 2011.
\newblock PhD thesis.

\bibitem{ivanovs_killing}
J.~Ivanovs.
\newblock A note on killing with applications in risk theory.
\newblock {\em Insurance: Mathematics and Economics}, 52(1):29 -- 34, 2013.

\bibitem{iva_palm}
J.~Ivanovs and Z.~Palmowski.
\newblock Occupation densities in solving exit problems for {M}arkov additive
  processes and their reflections.
\newblock {\em Stochastic Process. Appl.}, 122(9):3342--3360, 2012.

\bibitem{suprun}
V.~S. Koroljuk, V.~N. Suprun, and V.~M. \v{S}urenkov.
\newblock A potential-theoretic method in boundary problems for processes with
  independent increments and jumps of the same sign.
\newblock {\em Teor. Verojatnost. i Primenen.}, 21(2):253--259, 1976.

\bibitem{scale_review}
A.~Kuznetsov, A.~E. Kyprianou, and V.~Rivero.
\newblock The theory of scale functions for spectrally negative {L}\'{e}vy
  processes.
\newblock In {\em L\'{e}vy matters {II}}, volume 2061 of {\em Lecture Notes in
  Math.}, pages 97--186. Springer, Heidelberg, 2012.

\bibitem{kyprianou}
A.~E. Kyprianou.
\newblock {\em Introductory lectures on fluctuations of {L}\'{e}vy processes
  with applications}.
\newblock Universitext. Springer-Verlag, Berlin, 2006.

\bibitem{mordecki}
A.~L. Lewis and E.~Mordecki.
\newblock Wiener-{H}opf factorization for {L}\'{e}vy processes having positive
  jumps with rational transforms.
\newblock {\em J. Appl. Probab.}, 45(1):118--134, 2008.

\bibitem{loeffen14}
R.~L. Loeffen, J.-F. Renaud, and X.~Zhou.
\newblock Occupation times of intervals until first passage times for
  spectrally negative {L}\'{e}vy processes.
\newblock {\em Stochastic Process. Appl.}, 124(3):1408--1435, 2014.

\bibitem{oscar}
O.~Peralta–Guti\'errez.
\newblock {\em Advances of matrix–analytic methods in risk modelling}.
\newblock 2018.
\newblock PhD thesis, Technical University of Denmark.

\bibitem{sato}
K.-i. Sato.
\newblock {\em L\'{e}vy processes and infinitely divisible distributions},
  volume~68 of {\em Cambridge Studies in Advanced Mathematics}.
\newblock Cambridge University Press, Cambridge, 2013.
\newblock Translated from the 1990 Japanese original, Revised edition of the
  1999 English translation.

\bibitem{schilling}
R.~L. Schilling, R.~Song, and Z.~Vondra\v{c}ek.
\newblock {\em Bernstein functions}, volume~37 of {\em De Gruyter Studies in
  Mathematics}.
\newblock Walter de Gruyter \& Co., Berlin, second edition, 2012.
\newblock Theory and applications.

\bibitem{badescu}
D.~A. Stanford, F.~Avram, A.~L. Badescu, L.~Breuer, A.~da~Silva~Soares, and
  G.~Latouche.
\newblock Phase-type approximations to finite-time ruin probabilities in the
  {S}parre-{A}ndersen and stationary renewal risk models.
\newblock {\em Astin Bull.}, 35(1):131--144, 2005.

\bibitem{widder}
D.~V. Widder.
\newblock {\em The {L}aplace {T}ransform}.
\newblock Princeton Mathematical Series, v. 6. Princeton University Press,
  Princeton, N. J., 1941.

\end{thebibliography}

\end{document}